\documentclass{article}

\usepackage{arxiv}

\usepackage[utf8]{inputenc} 
\usepackage[T1]{fontenc}    
\usepackage{hyperref}       
\usepackage{url}            
\usepackage{booktabs}       
\usepackage{amsfonts}       
\usepackage{nicefrac}       
\usepackage{microtype}      
\usepackage{graphicx}
\usepackage{svg}
\usepackage{amsmath}
\usepackage{amsthm}
\newtheorem{thm}{Theorem}[section]

\newtheorem{prop}[thm]{Proposition}
\newtheorem{cor}[thm]{Corollary}
\theoremstyle{definition}
\newtheorem{exmp}{Example}[section]
\theoremstyle{remark}

\theoremstyle{definition}
\newtheorem{defn}{Definition}[section]

\title{Configurations of higher orders}

\author{
  Benjamin Peet\\
  Department of Mathematics\\
  St. Martin's University\\
  Lacey, WA 98503 \\
  \texttt{bpeet@stmartin.edu} \\
}

\begin{document}
\maketitle

\begin{abstract}
This paper begins by extending the notion of a combinatorial configuration of points and lines to a combinatorial configuration of points and planes that we refer to as configurations of order $2$. We then proceed to investigate a further extension to the notion of points and $k$-planes ($k$-dimensional hyperplanes) which we refer to as configurations of order $k$. We present a number of general examples such as stacked configurations of order $k$ - intuitively layering lower order configurations - and product configurations of order $k$. We discuss many analogues of standard configurations such as dual configurations, isomorphisms, graphical representations, and when a configuration is geometric. We focus mostly on configurations of order $2$ and specifically compute the number of possible symmetric configurations of order $2$ when each plane contains $3$ points for small values on $n$ - the total number of points in the configuration.
\end{abstract}


\section{Introduction}

Much has been written in the literature (both research and recreational) about configurations. These are given by the combinatorial definition of:

\begin{defn}
A configuration is a pair of sets of sets $(\mathcal{P},\mathcal{L})$ where $\mathcal{P}=\{p_{1}, \ldots , p_{n}\} $ are called points  and $\mathcal{L}=\{l_{1}, \ldots ,l_{m}\}$ are called lines with each $l_{i}$ is a union of sets from $\mathcal{P}$ and for any pair $p_{i_{1}},p_{i_{2}}$ there is at most one line $l_{j}$ that contains both $p_{i_{1}},p_{i_{2}}$ as subsets. Furthermore, each point is incident with the same number of lines as any other ($s$) and each line is incident with the same number of points as any other ($t$).
\end{defn}

Here it is assumed that both $s,t\geq 2$ and incidence of $p_{i}$ and $l_{j}$ means that $p_{i} \subset l_{j}$. For more details, see the foundational text of Grunbaum \cite{grunbaum2009configurations}.

Note the use of sets as points. This is entirely equivalent, but will become useful when finding representations of configurations of higher orders. For the purposes of this paper, each $p_{i}$ will be taken as distinct singleton sets. However, for future purposes outlined in the final section, it is useful to consider in generality.

Given the terminology of points and lines, we consider the natural progression to points and planes. This does not seem to have appeared in the literature in full generality (for example \cite{glynn2007note} considers only symmetric cases), so we make some explorations here by defining the notion of a configuration of order $2$ and then extending onwards to configurations of order $k$.

We continue some basic concepts from configurations including graphical representations, isomorphisms, and dual spaces. We generate some foundational results about these higher order configurations including some relations. We also do some computations to assess the possible symmetric configurations of order $2$ with $s=3$ along with their automorphism groups for $4\leq n \leq8$.

\section{Definition of a configuration of higher order}

We begin right away with the definition of a configuration of order $2$:

\begin{defn}
A configuration of order $2$ is a pair of sets of sets $(\mathcal{P},\mathcal{E})$ where $\mathcal{P}=\{p_{1}, \ldots , p_{n}\} $ are called points and $\mathcal{E}=\{e_{1}, \ldots ,e_{m}\}$ are called planes with each $e_{i}$ a union of elements of $\mathcal{P}$ and for any $3$ points there is at most one plane $e_{j}$ that has the union of all $3$ as a subset. Furthermore, each point is incident with the same number of planes as any other ($s$) and each plane is incident with the same number of points as any other ($t$).
\end{defn}

Here we assume that now $s,t\geq 3$ and again incidence of a point $p_{i}$ and a plane $e_{j}$ means that $p_{i} \subset e_{j}$. We must also further assume that there is some pair of points that are incident with more than one plane. Otherwise, this would simply be a configuration of order $1$.

We use $\mathcal{E}$ for the set of planes due to the German word \textit{Ebene} for plane.

We can then immediately define the notion of a configuration of order $k$:

\begin{defn}
A configuration of order $k$ is a pair of sets of sets $(\mathcal{P},\mathcal{E})$ where $\mathcal{P}=\{p_{1}, \ldots , p_{n}\} $ are called points and $\mathcal{E}=\{e_{1}, \ldots ,e_{m}\}$ are called $k$-planes with each $e_{i}$ a union of elements of $\mathcal{P}$ and for any $k+1$ points there is at most one plane $e_{j}$ that has the union of all $k+1$ as a subset. Furthermore, each point is incident with the same number of planes as any other ($s$) and each plane is incident with the same number of points as any other ($t$).
\end{defn}

We again here assume that there is some collection of $k$ points that are on more than one $k$-plane and that $s,t\geq k+1$

We say that a configuration of order $k$ is \textit{symmetric} if the number of points and the number of $k$-planes are equal.

Note that this generalizes the standard definition of a configuration. A standard configuration is a configuration of order $1$.

For convenience, we will sometimes make an intuitive exception to $s \geq k+1$ by using the term configuration of order $k$ \textit{without dual} if this is the case. This we will see is useful for explorations and in an intuitive way this seems reasonable. There certainly must be at least $k+1$ \textit{points} per $k$-plane, but a restriction on the number of $k$-\textit{planes} per point is necessary only if we wish to be able to construct the dual space (see later section). Hence configuration of order $k$ \textit{without dual}.

We note here that in \cite{grunbaum2009configurations} the term $k$-configuration is used. These are particular types of configuration (of order $1$) and not to be confused with configurations of order $k$.

\section{Examples and preliminary results}

We begin with an example - the simplest possible configuration of order $2$:

\begin{exmp}
Take four points, that is:

$$\mathcal{P}=\{\{1\},\{2\},\{3\},\{4\}\}$$

and then define the $2$-planes to be:

$$\mathcal{E}=\{\{1,2,3\},\{2,3,4\},\{3,4,1\},\{4,1,2\}\}$$

This defines a configuration of order $2$.

Each plane contains 3 points and each point is on 3 planes.

Note that:
$$4\times3=4\times3$$

We will see later as we begin to represent these configurations that this can be viewed as a tetrahedron.

\end{exmp}

We take this note from above to state in general:

\begin{prop}
Given $(\mathcal{P},\mathcal{E})$, a configuration of order $k$ with $n$ points; $m$ $k$-planes; each point incident with $s$ $k$-planes; and each $k$-plane incident with $t$ points, then the following hold:
\begin{enumerate}
    \item $ns=mt$
    \item $k < n-1$
\end{enumerate}
\end{prop}

\begin{proof}
For 1. we count the total incidences: each $k$-plane contains $t$ points, hence $mt$ distinct incidences. Each point is a member of $s$ $k$-planes so $ns$ distinct incidences. So necessarily, $ns=mt$.
For 2. we note that if $k \geq n-1$ then $t \geq k+1 \geq n$ which cannot be possible by definition.
\end{proof}

The next proposition gives the following:

\begin{prop}
Given $(\mathcal{P},\mathcal{E})$, a configuration of order $k$ with $n$ points; $m$ $k$-planes; each point incident with $s$ $k$-planes; and each $k$-plane incident with $t$ points, then the following hold:
\begin{enumerate}
    \item $m \geq 1+\frac{t(s-1)}{k}$
    \item $n \geq 1+\frac{s(t-1)}{k}$
\end{enumerate}
\end{prop}

\begin{proof}
Take a $k$-plane. Then each of the $t$ points on this $k$-plane lie on $s-1$ other $k$-planes. However, any $k+1$ points can lie on at most one $k$-plane, hence at a minimum, there are $1+\frac{t(s-1)}{k}$ $k$-planes.

A similar argument yields that there are at a minimum $1+\frac{s(t-1)}{k}$ points.
\end{proof}

\section{Basic concepts for configurations of order $k$}

We here give some basic concepts for configurations of order $k$ that are fairly straightforward extensions from configurations of order $1$.

The first is an isomorphism of configurations of order $k$:

\begin{defn}
Two configurations of order $k$ $(\mathcal{P}_{1},\mathcal{E}_{1})$ and $(\mathcal{P}_{1},\mathcal{E}_{1})$ are isomorphic if there is a bijection $f:\mathcal{P}_{1}\rightarrow \mathcal{P}_{2}$ such that if $e\in \mathcal{E}_{1}$, then $f(e)\in \mathcal{E}_{2}$.
\end{defn}

This leads us to the automorphism group of a configuration of order $k$:

\begin{defn}
The automorphism group of a configuration of order $k$ is the group of all automorphisms. That is, all self-isomorphisms.
\end{defn}

We now consider a dual configuration of configuration $k$:

\begin{defn}
A duality between configurations of order $k$ is an incidence preserving map that sends points to $k$-planes and $k$-planes to points. A configuration of configuration $k$ is self-dual if there is a duality from itself to itself.
\end{defn}

We give here an example of a self-dual configuration of order $2$.

\begin{exmp}
If we again take:
$$\mathcal{P}=\{\{1\},\{2\},\{3\},\{4\}\}$$
$$\mathcal{E}=\{\{1,2,3\},\{2,3,4\},\{1,3,4\},\{1,2,4\}\}$$

Which is the tetrahedron Example 3.1

Then the duality is given by:

$$\{1\}\mapsto \{1,2,3\}$$
$$\{2\} \mapsto \{2,3,4\}$$
$$\{3\} \mapsto \{1,3,4\}$$
$$\{4\} \mapsto \{1,2,4\}$$

To see that this is incidence preserving, note that the duality is exchanging the roles of set containment.
\end{exmp}

We remark here that if we allow $s\leq k$ as discussed earlier, the dual would necessarily have $t\leq k$ which is a clear contradiction. Hence in this case, we use the term configuration of order $k$ without dual.

\section{Geometric/topological representations of configurations of order $2$}

Given our experience from the example above, we now consider ways to visualize and represent these configuration of order $2$.

We noted that Example 3.1 can be viewed as a tetrahedron, and we consider this from an algebraic topology viewpoint as a simplicial complex. See \cite{hatcher2005algebraic}. However, in algebraic topology we would restrict ourselves to only allowing two faces per edge (in the case of $k=2$). There is no restriction here and we need to generalize such a representation.

We hence begin simply by considering configurations of order $2$ with $s=t=3$ and consider an abstract simplicial complex. That is:

\begin{defn}
A collection of non-empty finite subsets of a set $S$ is an abstract simplicial complex if, for every set $X$ in the collection, and every non-empty subset $Y \subset X$, $Y$ also belongs to the collection.
\end{defn}

For more information see \cite{lee2010introduction}.

This is therefore our first possible representation and we define as follows:

\begin{defn}
We say a configuration of order $2$ with $s=t=3$ is realized by an abstract simplicial complex by taking the collection $\mathcal{K}=\{P(e)|e\in\mathcal{E}\}$.
\end{defn}

Here $P(e)$ represents the power set and the dimension of an abstract simplicial complex is one less than the maximal cardinality of any set in the collection. Hence in this case dimension $2$.

Note that a line (simplex of dimension 1) as defined may not be an intersection of planes. However, in this restricted scenario of triangles, the lines will still be the boundary of the plane as defined. With $s=4$ and quadrilaterals, it could be that a line is not the boundary of a plane.

We give the following more complex example:

\begin{exmp}
$$\mathcal{P}=\{\{1\},\{2\},\{3\},\{4\}\,\{5\}\}$$
$$\mathcal{E}=\{\textit{all triples of points}\}$$

We have $n=5$, $m=5C3=10$, $s=6$, and naturally $t=3$. It is clear that this is not a standard simplicial complex, for example any edge between vertices lies on $3$ different faces (planes). 

We can however consider an abstract simplicial complex by $\mathcal{K}=\{\textit{all triples, all pairs, all points}\}$
\end{exmp}

In general, we could define a class of configurations of order $2$ by generalizing the above example to any number of points. Note that the lines along with the points form a complete graph. These could then be termed \textit{complete configurations of order $2$}.

We can then begin to realize these geometrically by embedding the points into Euclidean 3-space and taking the convex hull of the three points as a representation where two planes intersect geometrically if they share at least a point combinatorially.

In the case of the above example, we could embed the points with coordinates: 

$$(0,0,1),(0,0,-1),(-1,0,1),(1,0,1),(0,1,1)$$

It should be quickly observed however, that this definition will not hold if either $s$ or $t$ is greater than $3$ - as by the definition these abstract simplicial complexes would have dimension greater than 2.

We hence introduce an even more generalized notion of an abstract simplicial complex, by recognizing that the issue is that as $s$ and $t$ go beyond $3$ we no longer have triangles or $2$-simplices but instead polygons.

Hence we adapt a more general representation from \cite{mcmullen_schulte_2002}. We first state the original definition, simplifying by considering only rank 3:

\begin{defn}
An abstract polytope $\mathcal {K}$ of rank 3 is a partially ordered (by subset) finite collection of finite sets. Each element is referred to as a face and two faces are referred to as incident if one is a subset of the other (or vice versa). A chain of $\mathcal{K}$ is a totally ordered subset and the length of a chain is $i$ if the chain contains exactly $i+1$ elements. The maximal chains are called flags. For any two faces $F$ and $G$ of $\mathcal{K}$ with $F\subset G$, we call $G/F:=\{H|H\in\mathcal{K}, F\subset H\subset G\}$ a section of $\mathcal{K}$. Each section of $\mathcal{K}$ distinct from $\mathcal{K}$ is called a proper section.

Then $\mathcal{K}$ satisfies the following:

\begin{enumerate}
    \item$\mathcal{K}$ contains a least face and a greatest face; they are denoted by $F_{-1}$ and $F_{3}$.
    \item Each flag has length $4$. 
    \item Any section of $\mathcal{K}$ is connected, that is for any two proper faces $F$ and $G$ in the section there is a finite sequence of proper faces $F=H_{0},H_{1},\ldots,H_{k-1},H_{k}=G$ of $\mathcal{K}$ such that $H_{i-1},H_{i}$ are incident for $i=1,\ldots,k$.
    \item For each $i=0,1,\ldots,3$, if $F$ and $G$ are incident faces of $\mathcal{K}$ of ranks $i-1$ and $i+1$, then there are precisely 2 faces of rank $i$ such that $F<H<G$. 
\end{enumerate}
\end{defn}

We adjust this by not requiring the fourth condition (sometimes known as the diamond condition or homogeneity parameter). This allows more than two planes to meet at a line. We define this as a \textit{generalized abstract polytope}.

Then, given a configuration of order 2 we define:

$$F_{-1}=\emptyset$$
$$F_{3}=p_{1}\cup \ldots \cup p_{n}$$

and then the proper faces are initially:

$$p_{1},\ldots,p_{n}$$
$$e_{1},\ldots,e_{m}$$

That is, the elements of $\mathcal{P}$ and $\mathcal{E}$. These are respectively faces of rank $0$ and rank $2$.

It can be seen that we are missing faces of rank $1$. It remains to define these. We take the sets:

$$e_{i}\cap e_{j}$$

For any pair $e_{i},e_{j}\in\mathcal{E}$. Note that these intersections may be empty, in which $F_{-1}=\emptyset$ would be returned, or may be a point, in which case a face of rank $0$ would be returned.

It is clear now that this satisfies the requirements of Definition 5.3 (minus the fourth condition).

We now state the following:

\begin{defn}
A generalized abstract polytope of rank 3 can be geometrically realized in Euclidean $3$-space if the faces of rank $0$ can be embedded such that given a face of rank $2$ all the elements lie on a plane.
\end{defn}

Given that planes intersect planes in lines all elements of faces of rank $1$ will lie on lines.

We are now in a position to define when a configuration of order $2$ is realizable in Euclidean $3$-space:

\begin{defn}
A configuration of order $2$ is geometrically realizable in Euclidean $3$-space if its' associated generalized abstract polytope can be geometrically embedded in Euclidean $3$-space.
\end{defn}

We note here that for representative purposes a configuration of order $2$ that is not geometrically realizable in Euclidean $3$-space could be topologically represented similar to how the Fano plane is represented in the order $1$ case. We will see such an example in a following section.

We here given an example:

\begin{exmp}

The following is a configuration of order 2 without dual where $n=12$, $m=4$, $s=2$, and $t=6$.

$$\mathcal{P}=\{\{1\},\{2\},\{3\},\{4\},\{5\},\{6\},\{7\},\{8\},\{9\},\{10\},\{11\},\{12\}\}$$
$$\mathcal{E}=\{\{1,2,3,7,8,9\},\{3,4,5,9,10,11\},\{2,5,6,8,11,12\},\{1,4,6,7,10,12\}\}$$

As can be seen in Figure 1 this has been geometrically embedded in Euclidean $3$-space in such a way that the faces are flat.

\end{exmp}

\begin{figure}[!h]
\centering
\includegraphics[height=6cm]{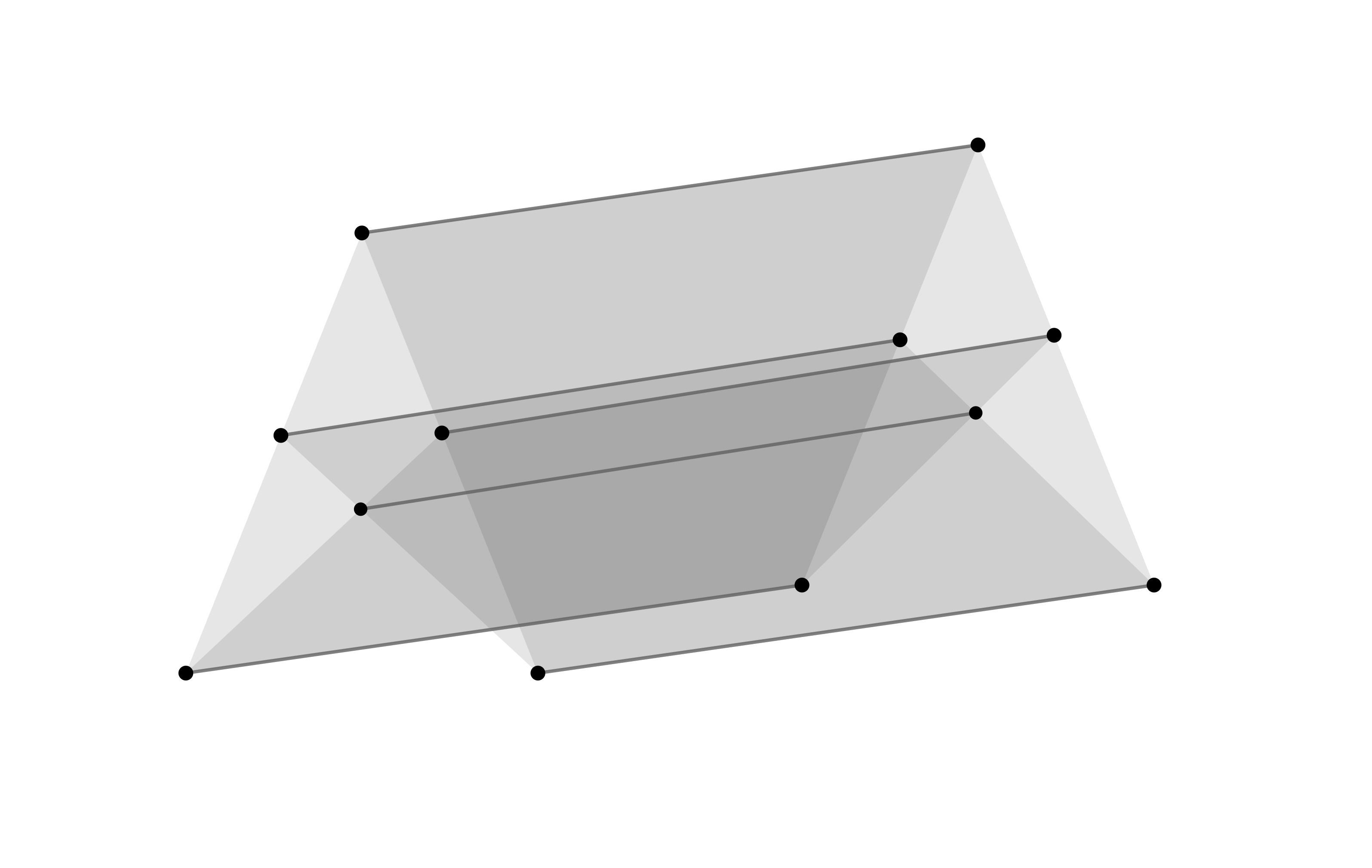}
\caption{Example of geometric embedding}
\end{figure}

Note that the lines shown are the faces of index 2, that is the intersections of the planes. Hence, there are no lines at the "ends".

Also note that two planes intersecting in the geometric realization does not imply that the intersection corresponds to a point. This is similar in the order $1$ case. Simply looking at a geometric realization of Pappus' configuration shows lines intersecting that do not refer to points. See once again \cite{grunbaum2009configurations} and Example 8.1 to see when planes intersect without referring to a point of the configuration of order $2$.

We here ask the question:

Question 1: What is the least number of points and points per plane so that a (symmetric) configuration of order 2 cannot be embedded geometrically?

Note that these methods could certainly be generalized for configurations of orders greater than $2$.

\section{Graphical representations of configurations of order $k$}

To consider a more graphical approach we introduce the following generalization of a Levi graph:

\begin{defn}
A Levi graph of a configuration of order $k$ is a bivalent graph, where points are denoted as black vertices and $k$-planes are denoted as white vertices. A point and a $k$-plane are incident if and only if there is an edge between the corresponding vertices.
\end{defn}

\begin{exmp}
We take the configuration of order $2$ from Example 3.1. As discussed previously, this can be topologically realized as a tetrahedron, or using the Levi Graph.
\end{exmp}

\begin{figure}[!h]
\centering
\includegraphics[height=6cm]{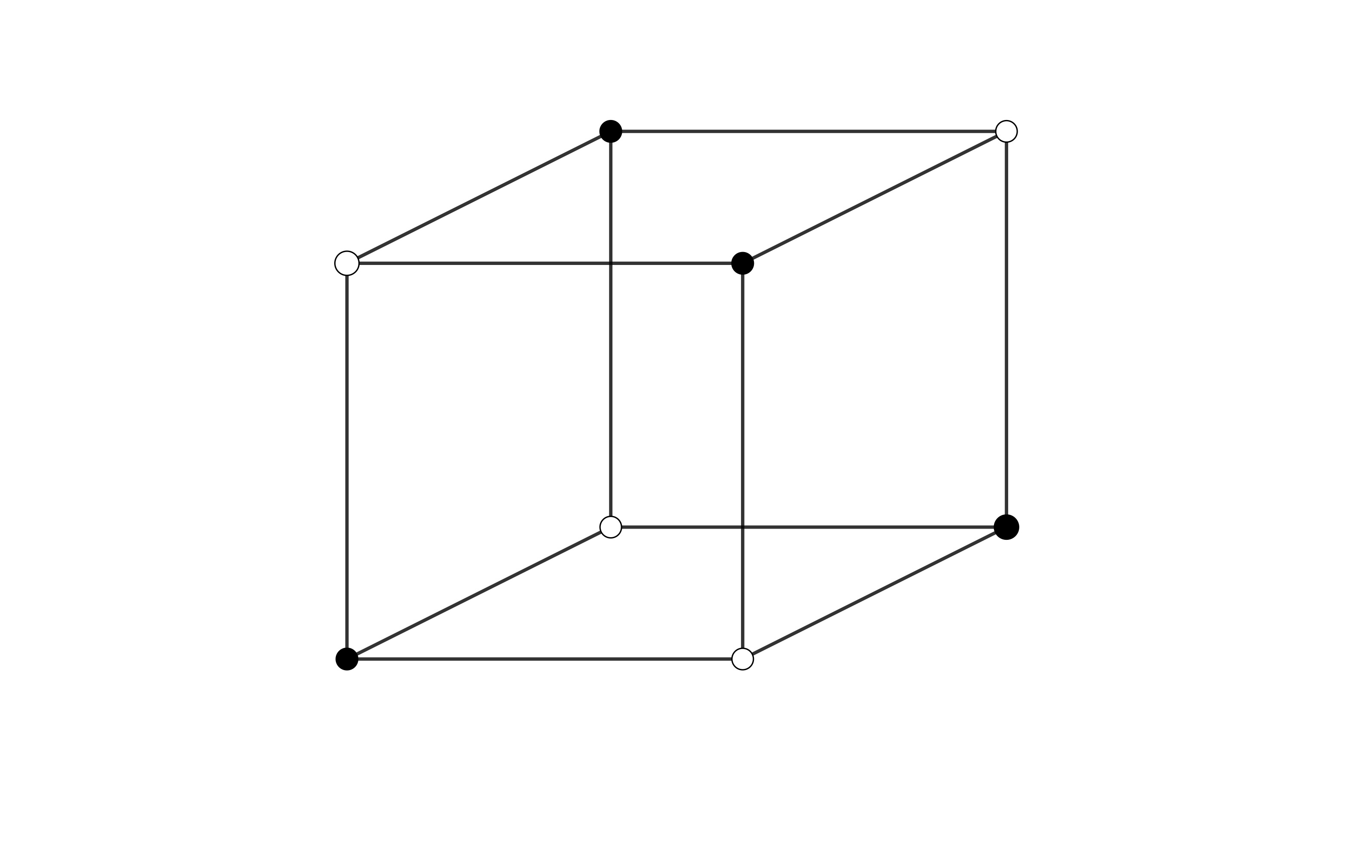}
\caption{Levi graph }
\end{figure}

Note that the Levi graph must satisfy the following:

\begin{prop}
For any $k+1$ black vertices there is at most one white vertex with which they all share an edge.
Each white vertex has valency $t$.
Each black vertex has valency $s$.
\end{prop}

\begin{proof}
This follows directly from the definitions.
\end{proof}

\section{Constructing configurations of dimension $2$ from configurations of order $1$}

Given that our work here is an extension of configurations of order $1$ we now begin to construct some different classes of configurations of higher order out of configurations of order $1$.

We first consider a class of configurations of order $2$ that we will refer to as stacked configurations.

\subsection{Simple stacked configurations}

\begin{defn}
Given a configuration of order $1$ given by $\mathcal{P}=\{p_{1},\ldots,p_{n}\}$ and $\mathcal{L}=\{l_{1},\ldots,l_{m}\}$ where $l_{i}=\{p_{k(i,1)},\ldots,p_{k(i,t)}\}$ we define a simple stacked configuration of order $2$ by $\mathcal{P^{\prime}}=\{p^{\prime}_{1,1},p^{\prime}_{1,2},\ldots,p^{\prime}_{n,1},p^{\prime}_{n,2}\}$ and $\mathcal{E}=\{e_{1},\ldots,e_{m}\}$ with:

$$e_{i}=\{p'_{k(i,1),1},p'_{k(i,1),2},\ldots,p'_{k(i,t),1},p'_{k(i,t),2}\}$$

\end{defn}

Simply put, we replace each of the points with $2$ points. This has the effect of "stacking" the configuration of order $1$. Note that there is a surjection from $\mathcal{P}'$ to $\mathcal{P}$ given by $p'_{i,j}\mapsto p_{i}$. This can be thought of as a projection of the configuration of order $2$ onto the original configuration of order $1$.

We should here verify that this definition and construction does indeed create a well-defined configuration of order $2$:

\begin{prop}
The above definition is well-defined. Furthermore, if the configuration of order $1$ is of type $(n_{s},m_{t})$, then the simple stacked configuration of order $2$ will be of type $(2n_{s},m_{2t})$.
\end{prop}

\begin{proof}
Firstly, there are $2n$ points and $m$ planes. Furthermore, we observe that there is a one-to-one correspondence between the lines of the configuration of order $1$ and the planes of the proposed configuration of order $2$. 

We then pick any three distinct points $p'_{i_{1},x_{1}},p'_{i_{2},x_{2}},p'_{i_{3},x_{3}}$ and consider the projection back onto the original configuration of order $1$. This yields $p_{i_{1}},p_{i_{2}},p_{i_{3}}$ where at least one is distinct - $x_{1},x_{2},x_{3}\in\{1,2\}$ cannot be distinct by the pigeon hole principle. By the property of the configuration of order $1$, these are all on at most one line. Hence by the one-to-one correspondence, $p'_{i_{1},x_{1}},p'_{i_{2},x_{2}},p'_{i_{3},x_{3}}$ can all be on at most one plane.
\end{proof}

To see a very simple example of this, we create a degree $2$ stacked configuration of order $2$ of the $4$-line geometry by:

\begin{exmp}
$$\mathcal{P}=\{\{1\},\{2\},\{3\},\{4\},\{5\},\{6\}\}$$ and $$\mathcal{L}=\{\{1,2,3\},\{1,4,5\},\{3,4,6\},\{2,5,6\}\}$$ becomes stacked as:
$$\mathcal{P}=\{\{1\},\{7\},\{2\},\{8\},\{3\},\{9\},\{4\},\{10\},\{5\},\{11\},\{6\},\{12\}\}$$ and $$\mathcal{E}=\{\{1,7,2,8,3,9\},\{1,7,4,10,5,11\},\{3,9,4,10,6,12\},\{2,8,5,11,6,12\}\}$$
\end{exmp}

Note that the original configuration of order $1$ had $s_{1}=2$ and $t_{1}=3$ and then the stacked configuration of order $2$  has $s_{2}=2$ and $t_{2}=6$ as determined by the previous proposition. So we can state that this is only in fact a configuration of order $2$ without dual.

This is the geometric configuration of order $2$ represented in Figure 1.

The effect on the Levi graph of simple stacking is to simply "double" the points. 

\begin{exmp}
For a very simple example we see the Levi graph of a simple stacking of a $3$-point configuration in Figure 3. Note that again this is a configuration without dual.
\end{exmp}
\begin{figure}[!h]
\centering
\includegraphics[height=6cm]{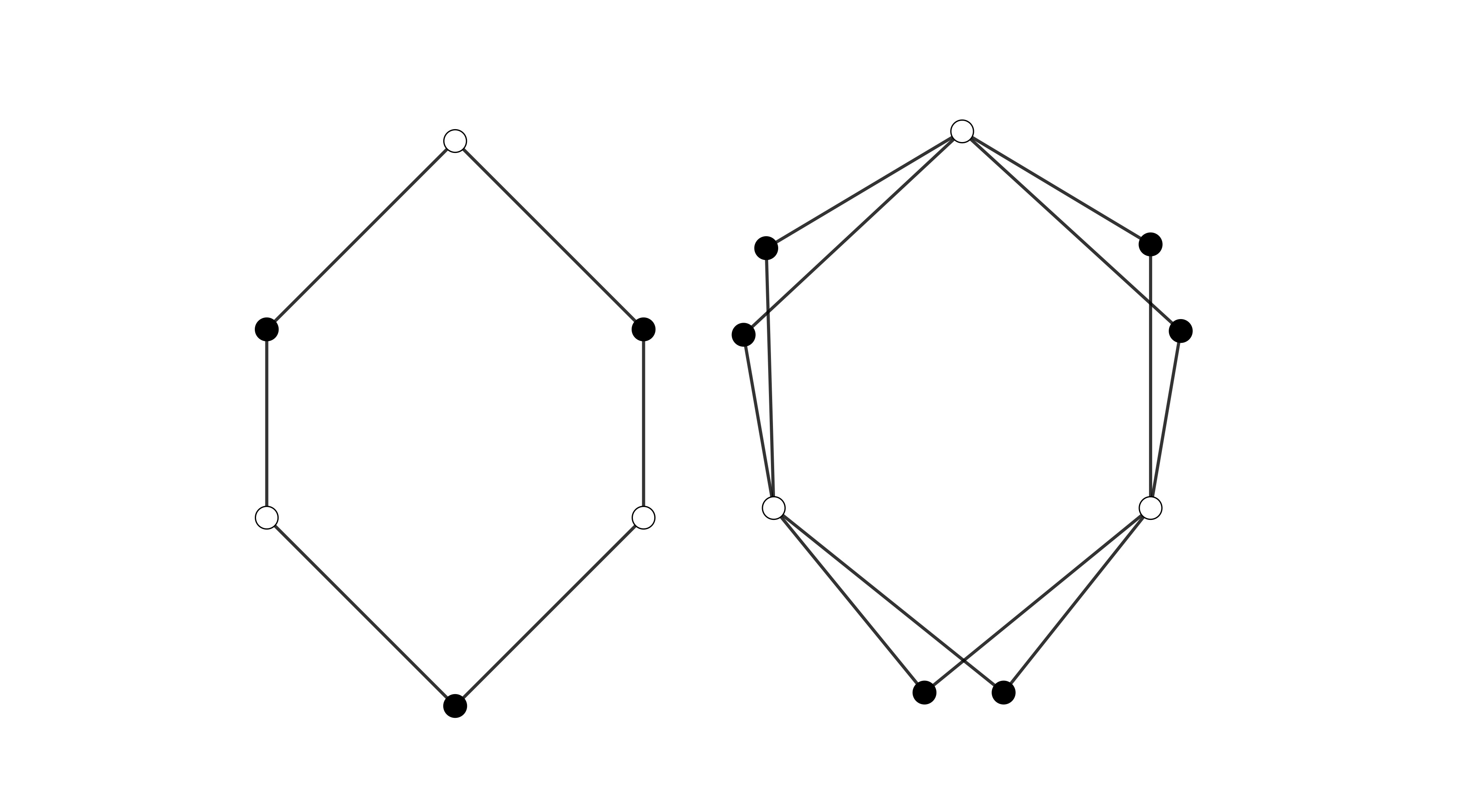}
\caption{Levi graph }
\end{figure}

\subsection{General stacked configurations}

We now note that we could in fact stack two different configurations of order $1$.

We give an example of such a configuration of order $2$ without dual:

\begin{exmp}
We take the square configuration of order $1$:
$$\mathcal{P}=\{\{1\},\{2\},\{3\},\{4\}\},\mathcal{L}=\{\{1,2\},\{2,3\},\{3,4\},\{1,4\}\}$$
and the $4$-line configuration:
$$\mathcal{P}_{2}=\{\{5\},\{6\},\{7\},\{8\},\{9\},\{10\}\},\mathcal{L}_{2}=\{\{5,6,7\},\{5,8,9\},\{6,9,10\},\{7,8,10\}\}$$

From this, we can take the points as the points from both the configurations of order $1$ and the planes as the union of a line from each of the two configurations of order $1$. That is:

$$\mathcal{P}=\{\{1\},\{2\},\{3\},\{4\},\{5\},\{6\},\{7\},\{8\},\{9\},\{10\}\}, \mathcal{E}=\{\{1,2,5,6,7\},\{2,3,5,8,9\},\{3,4,6,9,10\},\{1,4,7,8,10\}$$

It can be seen that each plane has $5$ points and that each point occurs on $2$ planes. Finally, it can be verified that any three points occurs on at most one plane and that there are pairs of points that occur on two planes. This is therefore a configuration of order $2$ without dual.

Note that the number of lines needed to be equal, and the number of lines per point needed to be equal for this construction to work.
\end{exmp}

We can now therefore define a general stacked configuration of order $2$ as follows:

\begin{defn}
A configuration of order $2$ is a general stacked configuration of order $2$ if the points can be divided into two sets $\mathcal{P}_{1}$ and $\mathcal{P}_{2}$ so that by defining:
$$\mathcal{L}_{1}=\{E\cap P_{1} | E \in \mathcal{E}\}$$
$$\mathcal{L}_{2}=\{E\cap P_{2} | E \in \mathcal{E}\}$$
We have two configurations of order $1$ given by $(\mathcal{P}_{1},\mathcal{L}_{1})$ and $(\mathcal{P}_{2},\mathcal{L}_{2})$. Here $P_{1}$ and $P_{2}$ are the unions of the elements of $\mathcal{P}_{1}$ and $\mathcal{P}_{2}$ respectively.
\end{defn}

We can then state the following results:

\begin{prop}
Two configurations of order $1$ with values $(m_{s},n_{t})$ and $(m'_{s^{\prime}},n'_{t^{\prime}})$ respectively can be stacked only if:

\begin{enumerate}
    \item $m=m'$
    \item $s=s'$
\end{enumerate}
\end{prop}

\begin{proof}
Suppose $m\neq m'$ so then without loss of generality there is some $E\in \mathcal{E}$ so that $E \cap P_{1}=\emptyset$ and $E\subset P_{2}$. So then $E\in L_{2}$ which implies $t=t_{2}$ and as $t=t_{1}+t_{2}$ we must have $t_{1}=0$. This is our contradiction and so $m_{1}=m_{2}$.
\end{proof}

\begin{cor}
Suppoose that two configurations of order $1$ with values $(m_{s},n_{t})$ and $(m'_{s^{\prime}},n'_{t^{\prime}})$ respectively satisfy the requirements of Proposition 7.2, then the resultant general stacked configuration of order $2$ will have values:

\begin{enumerate}
    \item $m_{2}=m=m'$
    \item $s_{2}=s=s'$
    \item $n_{2}=n+n'$
    \item $t_{2}=t+t'$
\end{enumerate}
\end{cor}

\begin{proof}
This follows from Proposition 4.2. or directly from the definition.
\end{proof}

We remark here that you can very naturally extend the definitions to stacking any number of configurations of order $1$. 

We have seen that stacked configurations of order $2$ exist. It now remains to see whether all configurations of order $2$ are stacked. Firstly, any configuration of order $2$ with an prime number of points cannot be stacked, but we offer the complete configuration as in section 5 with six points as palpably not two 3-point configurations stacked.

\subsection{Stacked Fano configuration}

For a further example we here stack two Fano planes:

\begin{exmp}
\begin{figure}[!h]
\centering
\includegraphics[height=6cm]{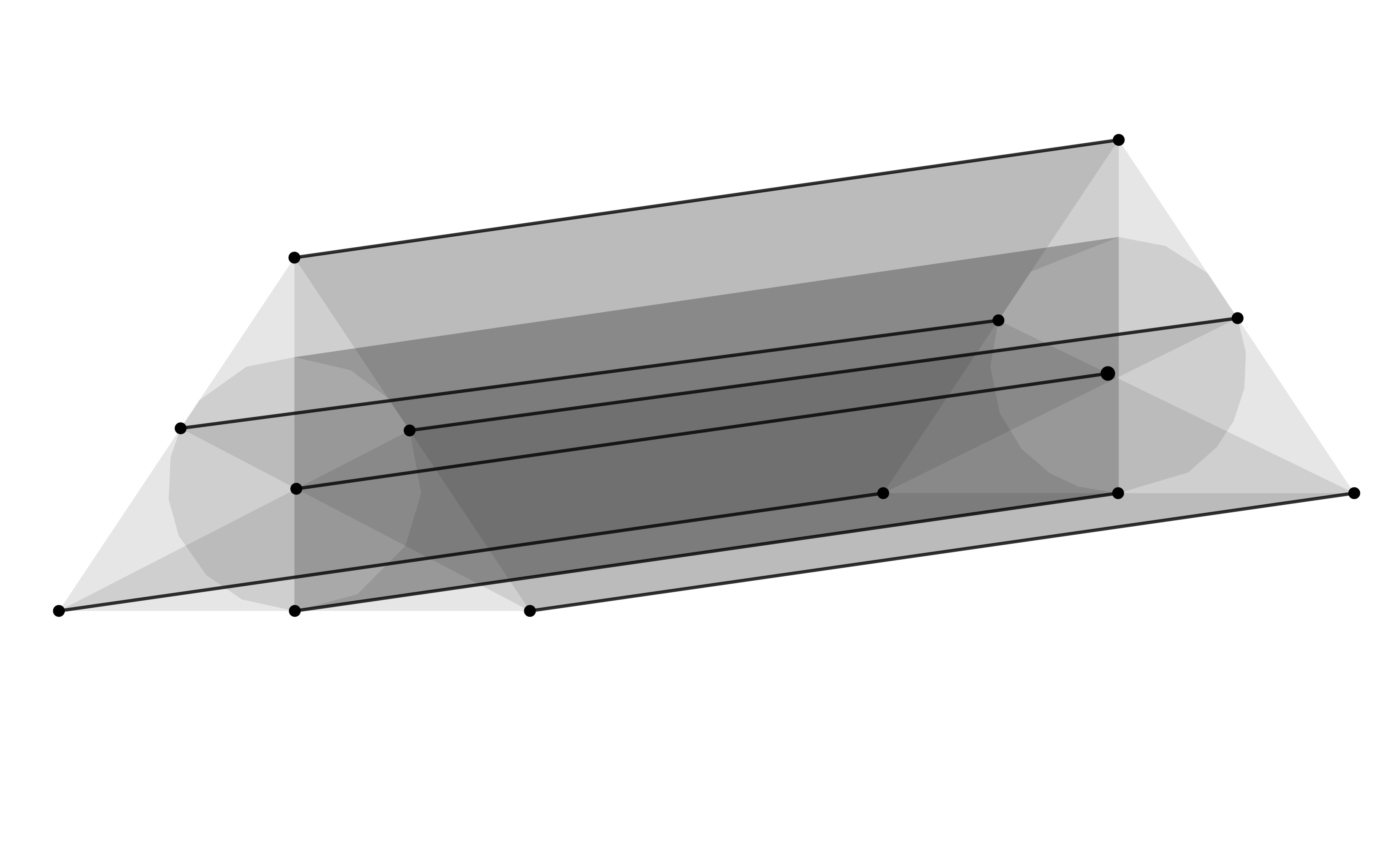}
\caption{Stacked Fano plane}
\end{figure}
\end{exmp}

The above figure shows a topological embedding - one "plane" is a cylinder. We now establish that this configuration of order $2$ is not geometric.

\begin{prop}
The stacked Fano plane is not geometric.
\end{prop}

\begin{proof}
Suppose a geometric realization does exist. Consider the set of intersections of these geometric planes. This consists of either geometric points or geometric lines. Then choose a plane that does not intersect the intersection points and does not contain the intersection lines but does intersect all of the geometric planes. Then this plane will contain seven lines of intersection with the geometric planes necessarily in a Fano configuration. This is a contradiction.
\end{proof}

\subsection{Product configurations}

We finish this section with another class of configurations of higher order that are truly \textit{products} by taking motivation from graph theory. In particular, we consider the Cartesian product (also known as box product). For more information on these graph theory topics see \cite{harary1969graph} or \cite{imrich2008topics}.

\begin{defn}
Let $(\mathcal{P}_{1},\mathcal{L}_{1})$ and $(\mathcal{P}_{2},\mathcal{L}_{2})$ be two configurations of order $1$, then the cartesian product of these two configurations is given by $(\mathcal{P},\mathcal{E})$ such that:
\begin{enumerate}
\item The point set is $\mathcal{P}=\mathcal{P}_{1}\times\mathcal{P}_{2}$
\item The planes are given by $\mathcal{E}=\{l_{1}\times l_{2}|l_{1}\in\mathcal{L}_{1},l_{2}\in\mathcal{L}_{2}\}$
\end{enumerate}
\end{defn}

It remains to once again check that this is well-defined:

\begin{prop}
The above definition is well-defined. Furthermore, if the configurations of order $1$ are of type $(n_{1}{}_{s_{1}},m_{1}{}_{t_{1}})$ and $(n_{2}{}_{s_{2}},m_{2}{}_{t_{2}})$, then the product configuration is of order $\text{max}\{t_{1},t_{2}\}$ and will be of type $(n_{1}n_{2}{}_{s_{1}s_{2}},m_{1}m_{2}{}_{t_{1}t_{2}})$.
\end{prop}

\begin{proof}
It is clear that there must be now $n_{1}n_{2}$ points and $m_{1}m_{2}$ planes by the following bijections $f$ and $g$. 

$$f:\mathcal{P}\rightarrow \mathcal{P}_{1}\times \mathcal{P}_{2}$$

$$g:\mathcal{E}\rightarrow \mathcal{L}_{1}\times \mathcal{L}_{2}$$

Each plane necessarily contains now $t_{1}t_{2}$ points and each point is on $s_{1}s_{2}$ planes.

Let $\textit{max}\{t_{1},t_{2}\}=t$.

We then take points $a_{1},\ldots,a_{t+1}\in\mathcal{P}$. If $f_{1}$ and $f_{2}$ are the projections onto $\mathcal{P}_{1}$ and $\mathcal{P}_{2}$ respectively, then by necessity we must have that the collections:

$$f_{1}(a_{1}),\ldots,f_{1}(a_{t+1})$$

and

$$f_{2}(a_{1}),\ldots,f_{2}(a_{t+1})$$

contain only $t_{1}$ distinct points in $\mathcal{P}_{1}$ and $t_{2}$ points in $\mathcal{P}_{2}$ respectively.

But we know that each plane contains exactly $t_{1}t_{2}$ points, hence if two planes contain all these $t+1$ points, they must be equal.

\end{proof}

\begin{cor}
The only product configurations of order $2$ have $t=4$.
\end{cor}

\begin{proof}
This follows as $\text{max}\{t,t'\}=2$ implies $t=t'=2$.
\end{proof}

\begin{cor}
The only symmetric product configurations of order $2$ are product configurations of polygons.
\end{cor}

\begin{proof}
This follows immediately from the previous corollary.
\end{proof}

To finish this section we give the following propositions and question regarding these constructed configurations of higher orders.

\begin{prop}
If a configuration of order $1$ is geometric then so is any simple stacked configuration of order $2$ of it.
\end{prop}

\begin{proof}
This is a fairly straight forward exercise in defining the planes to be the lines cross $\mathbb{R}$.
\end{proof}

Some valid questions might be:

Question 2: Are all general stacked configurations order 2 geometric if each configuration of order 1 is geometric?

Question 3: If two configurations of order $1$ are geometric then is the product configuration of order $2$ also geometric?

\section{Calculating symmetric configurations of order $2$ with $s=3$}

We here take the time to generate some specific configurations of order $2$ in the symmetric case with $s=3$. These are generated using the GAP algebra coding language and the specific code is in a github repository available at \cite{Peet_Configurations_of_higher_2022}. Table 1 (which follows the references) shows all possible configuration of order $2$ for $4\leq n \leq 8$.

The analogous case for configurations of order $1$ would be symmetric configurations with $s=2$. These can be viewed simply as polygons. It could be suggested that then the symmetric configurations of order $2$ with $s=3$ will able to be viewed as triangulated surfaces. 

We will see that this is not true and give the following example to illustrate:

\begin{exmp}
$$\mathcal{P}=\{\{1\},\{2\},\{3\},\{4\},\{5\},\{6\},\{7\},\{8\},\{9\}\}$$
$$\mathcal{E}=\{\{1,2,3\},\{1,2,5\},\{1,2,5\},\{3,6,7\},\{4,6,7\},\{5,6,7\},\{3,8,9\},\{4,8,9\},\{5,8,9\}\}$$

This has $m=n=9$. As you can see in Figure 5, this is certainly not a triangulated surface as each of the lines shown meet three planes. Note that provided the three lines shown are not coplanar this is geometric. Also note that although there are planes intersecting they are only referring to one point in the configuration, not at least two as might be expected. 
\end{exmp}

\begin{figure}[!h]
\centering
\includegraphics[height=6cm]{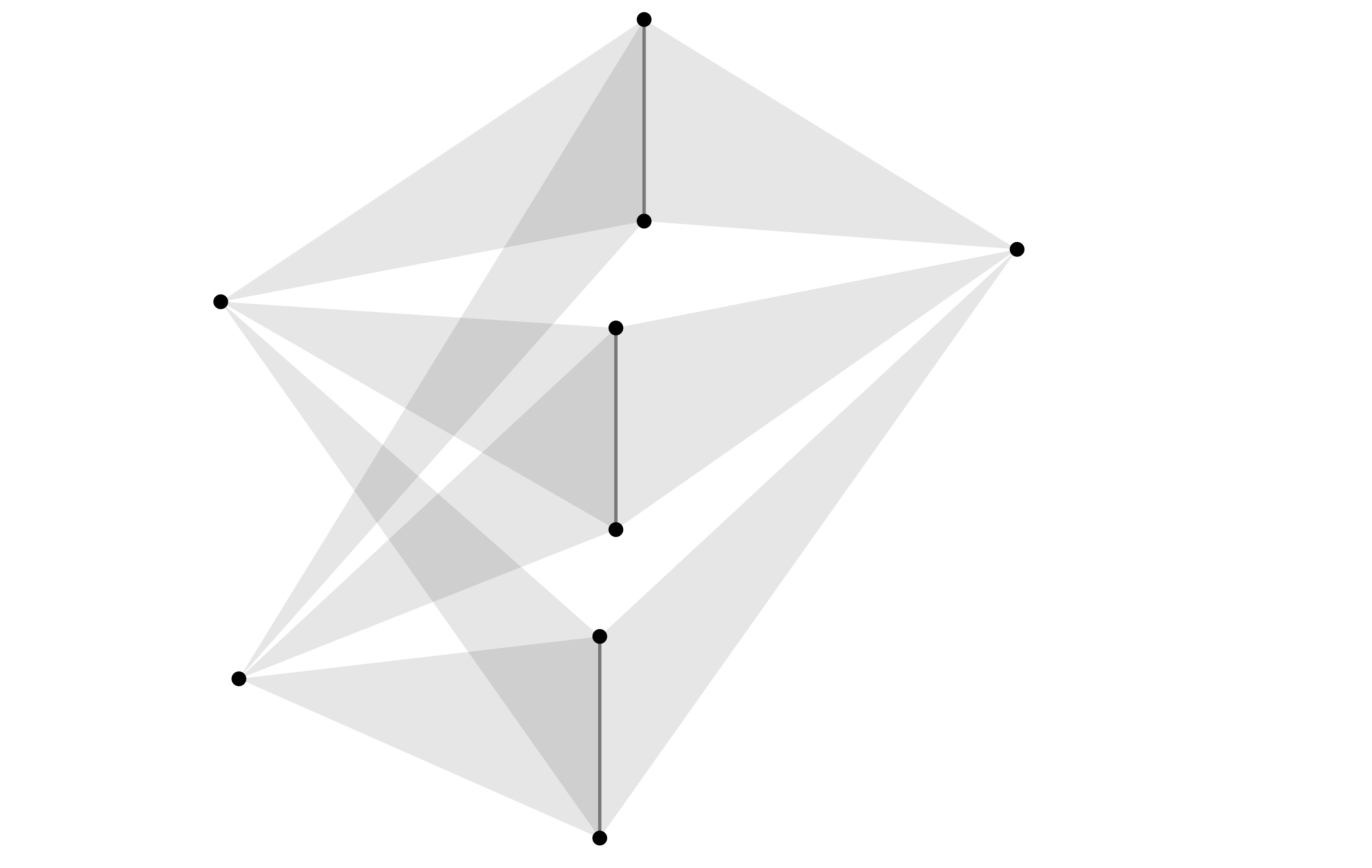}
\caption{Example 8.1}
\end{figure}

\begin{table}[h]
\begin{center}
\begin{math}
\begin{array}{ | l | l | l | l |}
\hline
    \text{Name} &  n  & \text{Planes}  & \text{Automorphism group}\\ \hline
     4.1 & 4 & \{1,2,3\},\{1,2,4\},\{1,3,4\},\{2,3,4\}   &  S_{4} \\ \hline
     5.1 & 5 & \{1,2,5\},\{1,3,4\}, \{1,4,5\},\{2,3,4\},\{2,3,5\}  &  D_{5} \\ \hline
     6.1 & 6    & \{1,2,3\},\{1,2,4\},\{1,3,5\},\{2,4,6\},\{3,5,6\},\{4,5,6\}& D_{6} \\ \hline
     6.2 & 6    &\{1,2,3\},\{1,2,4\},\{1,3,5\},\{2,5,6\},\{3,4,6\},\{4,5,6\} & \mathbb{Z}_{4} \\ \hline
     6.3 & 6   &\{1,2,3\},\{1,2,4\},\{1,5,6\},\{2,5,6\},\{3,4,5\},\{3,4,6\} & \mathbb{Z}_{2} \times A_{4}\\\hline
      7.1 & 7   & \{1, 4, 7 \}, \{ 1, 5, 6 \}, \{1, 6, 7 \}, \{ 2, 3, 6 \}, \{ 2, 3, 7 \}, 
      \{ 2, 4, 5 \}, \{ 3, 4, 5 \}  & \mathbb{Z}_{2}\times\mathbb{Z}_2 \\\hline
      7.2 & 7   &  \{ 1, 4, 7 \}, \{ 1, 5, 7 \}, \{ 1, 6, 7 \}, \{ 2, 3, 5 \}, \{ 2, 3, 6 \}, 
      \{ 2, 4, 6 \}, \{ 3, 4, 5 \}  & \mathbb{Z}_{2}\times\mathbb{Z}_2 \\\hline
      7.3 & 7   & \{ 1, 5, 6 \}, \{ 1, 5, 7 \}, \{ 1, 6, 7 \}, \{ 2, 3, 4 \}, \{ 2, 3, 7 \}, 
      \{ 2, 4, 6 \}, \{ 3, 4, 5 \} & S_{3} \\\hline
      7.4 & 7   & \{1, 4, 7 \}, \{ 1, 5, 6 \}, \{ 1, 6, 7 \}, \{ 2, 3, 5 \}, \{ 2, 3, 7 \}, 
      \{ 2, 4, 6 \}, \{ 3, 4, 5 \}  & \mathbb{Z}_{2} \\\hline
      7.5 & 7   &  \{ 1, 4, 6 \}, \{ 1, 5, 7 \}, \{ 1, 6, 7 \}, \{ 2, 3, 5 \}, \{ 2, 3, 7 \}, 
      \{ 2, 4, 6 \}, \{ 3, 4, 5 \}  & \mathbb{Z}_{2} \\\hline
      7.6 & 7   & \{ 1, 3, 7 \}, \{ 1, 4, 7 \}, \{ 1, 5, 7 \}, \{ 2, 3, 6 \}, \{ 2, 4, 6 \}, 
      \{ 2, 5, 6 \}, \{ 3, 4, 5 \} & D_{4}\times S_{3} \\\hline
      7.7 & 7   & \{ 1, 3, 6 \}, \{ 1, 4, 7 \}, \{ 1, 5, 7 \}, \{ 2, 3, 7 \}, \{ 2, 4, 6 \}, 
      \{ 2, 5, 6 \}, \{ 3, 4, 5 \} & \mathbb{Z}_{2}\times \mathbb{Z}_{2} \times \mathbb{Z}_{2} \\\hline
      7.8 & 7   &  \{ 1, 3, 5 \}, \{ 1, 4, 7 \}, \{ 1, 6, 7 \}, \{ 2, 3, 7 \}, \{ 2, 4, 6 \}, 
      \{ 2, 5, 6 \}, \{ 3, 4, 5 \}  & \mathbb{Z}_{3} \\\hline
       7.9 & 7   &  \{ 1, 4, 5 \}, \{ 1, 5, 7 \}, \{ 1, 6, 7 \}, \{ 2, 3, 4 \}, \{ 2, 3, 6 \}, 
      \{ 2, 6, 7 \}, \{ 3, 4, 5 \}  & D_{7} \\\hline

      8.1 & 8   &   \{ 1, 5, 8 \}, \{ 1, 6, 8 \}, \{ 1, 7, 8 \}, \{ 2, 3, 7 \}, \{ 2, 4, 7 \}, \{ 2, 5, 6 \}, \{ 3, 4, 5 \}, \{ 3, 4, 6 \}  &  \mathbb{Z}_{2}\times \mathbb{Z}_{2} \times \mathbb{Z}_{2}\\\hline

      8.2 & 8   &  \{ 1, 5, 8 \}, \{ 1, 6, 7 \}, \{ 1, 7, 8 \}, \{ 2, 3, 8 \}, \{ 2, 4, 7 \}, \{ 2, 5, 6 \}, \{ 3, 4, 5 \}, \{ 3, 4, 6 \}   & \mathbb{Z}_{2} \\\hline

      8.3 & 8   &  \{ 1, 6, 7 \}, \{ 1, 6, 8 \}, \{ 1, 7, 8 \}, \{ 2, 3, 8 \}, \{ 2, 4, 5 \}, \{ 2, 5, 7 \}, \{ 3, 4, 5 \}, \{ 3, 4, 6 \}   &  \mathbb{Z}_{2}\\\hline
      
      8.4 & 8   &  \{ 1, 5, 8 \}, \{ 1, 6, 8 \}, \{ 1, 7, 8 \}, \{ 2, 3, 7 \}, \{ 2, 4, 6 \}, \{ 2, 5, 7 \}, \{ 3, 4, 5 \}, \{ 3, 4, 6 \}  & \mathbb{Z}_{2} \\\hline

      8.5 & 8   &   \{ 1, 5, 8 \}, \{ 1, 6, 7 \}, \{ 1, 7, 8 \}, \{ 2, 3, 8 \}, \{ 2, 4, 6 \}, \{ 2, 5, 7 \}, \{ 3, 4, 5 \}, \{ 3, 4, 6 \}  &  \{\}\\\hline
      
      8.6 & 8   &  \{ 1, 5, 7 \}, \{ 1, 6, 8 \}, \{ 1, 7, 8 \}, \{ 2, 3, 8 \}, \{ 2, 4, 6 \}, \{ 2, 5, 7 \}, \{ 3, 4, 5 \}, \{ 3, 4, 6 \}  &  \{\}\\\hline
      
      8.7 & 8   &  \{ 1, 5, 8 \}, \{ 1, 6, 7 \}, \{ 1, 6, 8 \}, \{ 2, 3, 8 \}, \{ 2, 4, 7 \}, \{ 2, 5, 7 \}, \{ 3, 4, 5 \}, \{ 3, 4, 6 \}   & \{\} \\\hline
      
      8.8 & 8   &  \{ 1, 5, 6 \}, \{ 1, 6, 8 \}, \{ 1, 7, 8 \}, \{ 2, 3, 8 \}, \{ 2, 4, 7 \}, \{ 2, 5, 7 \}, \{ 3, 4, 5 \}, \{ 3, 4, 6 \}  &  \{\}\\\hline
       
      8.9 & 8   &  \{ 1, 5, 7 \}, \{ 1, 6, 7 \}, \{ 1, 6, 8 \}, \{ 2, 3, 8 \}, \{ 2, 4, 8 \}, \{ 2, 5, 7 \}, \{ 3, 4, 5 \}, \{ 3, 4, 6 \}  &  \mathbb{Z}_{2} \times \mathbb{Z}_{2}\\\hline

      8.10 & 8   &  \{ 1, 5, 6 \}, \{ 1, 6, 7 \}, \{ 1, 7, 8 \}, \{ 2, 3, 8 \}, \{ 2, 4, 8 \}, \{ 2, 5, 7 \}, \{ 3, 4, 5 \}, \{ 3, 4, 6 \}  & \mathbb{Z}_{2} \\\hline

      8.11 & 8   &    \{ 1, 4, 6 \}, \{ 1, 6, 8 \}, \{ 1, 7, 8 \}, \{ 2, 3, 7 \}, \{ 2, 5, 7 \}, \{ 2, 5, 8 \}, \{ 3, 4, 5 \}, \{ 3, 4, 6 \}  & \mathbb{Z}_{2} \\\hline
  
      8.12 & 8   &    \{ 1, 4, 7 \}, \{ 1, 6, 7 \}, \{ 1, 6, 8 \}, \{ 2, 3, 8 \}, \{ 2, 5, 7 \}, \{ 2, 5, 8 \}, \{ 3, 4, 5 \}, \{ 3, 4, 6 \}  & \mathbb{Z}_{2} \\\hline
      
      8.13 & 8   &    \{ 1, 4, 6 \}, \{ 1, 6, 8 \}, \{ 1, 7, 8 \}, \{ 2, 3, 7 \}, \{ 2, 5, 6 \}, \{ 2, 5, 8 \}, \{ 3, 4, 5 \}, \{ 3, 4, 7 \}  & \mathbb{Z}_{2} \\\hline
  
      8.14 & 8   &    \{ 1, 3, 8 \}, \{ 1, 4, 8 \}, \{ 1, 5, 8 \}, \{ 2, 5, 6 \}, \{ 2, 5, 7 \}, \{ 2, 6, 7 \}, \{ 3, 4, 6 \}, \{ 3, 4, 7 \}  & \mathbb{Z}_{2} \times \mathbb{Z}_{2} \times \mathbb{Z}_{2} \\\hline
  
      8.15 & 8   &   \{ 1, 4, 5 \}, \{ 1, 5, 7 \}, \{ 1, 6, 8 \}, \{ 2, 3, 8 \}, \{ 2, 5, 8 \}, \{ 2, 6, 7 \}, \{ 3, 4, 6 \}, \{ 3, 4, 7 \}  &  \{\}\\\hline
      
      8.16 & 8   &     \{ 1, 4, 5 \}, \{ 1, 5, 6 \}, \{ 1, 7, 8 \}, \{ 2, 3, 8 \}, \{ 2, 5, 7 \}, \{ 2, 6, 8 \}, \{ 3, 4, 6 \}, \{ 3, 4, 7 \}  & \mathbb{Z}_{2} \times \mathbb{Z}_{2} \\\hline
      
      8.17 & 8   &       \{ 1, 3, 7 \}, \{ 1, 4, 7 \}, \{ 1, 5, 8 \}, \{ 2, 5, 6 \}, \{ 2, 5, 8 \}, \{ 2, 6, 8 \}, \{ 3, 4, 6 \}, \{ 3, 4, 7 \}  & D_{4} \\\hline
      
      8.18 & 8   &      \{ 1, 3, 7 \}, \{ 1, 4, 6 \}, \{ 1, 5, 8 \}, \{ 2, 5, 7 \}, \{ 2, 5, 8 \}, \{ 2, 6, 8 \}, \{ 3, 4, 6 \}, \{ 3, 4, 7 \}  &  \mathbb{Z}_{2} \\\hline

      8.19 & 8   &        \{ 1, 3, 4 \}, \{ 1, 5, 8 \}, \{ 1, 6, 7 \}, \{ 2, 5, 7 \}, \{ 2, 5, 8 \}, \{ 2, 6, 8 \}, \{ 3, 4, 6 \}, \{ 3, 4, 7 \}  & D_{6} \\\hline
      
      8.20 & 8   &        \{ 1, 3, 4 \}, \{ 1, 5, 6 \}, \{ 1, 7, 8 \}, \{ 2, 5, 7 \}, \{ 2, 5, 8 \}, \{ 2, 6, 8 \}, \{ 3, 4, 6 \}, \{ 3, 4, 7 \}  & \mathbb{Z}_{2} \times \mathbb{Z}_{2} \\\hline
      
      8.21 & 8   &       \{ 1, 3, 4 \}, \{ 1, 5, 7 \}, \{ 1, 5, 8 \}, \{ 2, 5, 8 \}, \{ 2, 6, 7 \}, \{ 2, 6, 8 \}, \{ 3, 4, 6 \}, \{ 3, 4, 7 \}  & \mathbb{Z}_{2} \times \mathbb{Z}_{2} \\\hline
      
      8.22 & 8   &        \{ 1, 4, 6 \}, \{ 1, 5, 7 \}, \{ 1, 5, 8 \}, \{ 2, 3, 8 \}, \{ 2, 5, 7 \}, \{ 2, 6, 7 \}, \{ 3, 4, 6 \}, \{ 3, 4, 8 \}  &  \mathbb{Z}_{4}\\\hline
      
      8.23 & 8   &       \{ 1, 3, 7 \}, \{ 1, 4, 5 \}, \{ 1, 7, 8 \}, \{ 2, 4, 8 \}, \{ 2, 5, 8 \}, \{ 2, 6, 7 \}, \{ 3, 4, 6 \}, \{ 3, 5, 6 \}  & \mathbb{Z}_{2} \times \mathbb{Z}_{2} \\\hline
      
      8.24 & 8   &        \{ 1, 4, 8 \}, \{ 1, 5, 7 \}, \{ 1, 6, 8 \}, \{ 2, 3, 8 \}, \{ 2, 4, 6 \}, \{ 2, 5, 7 \}, \{ 3, 4, 7 \}, \{ 3, 5, 6 \}  &  \mathbb{Z}_{2}\\\hline
      
      8.25 & 8   &       \{ 1, 4, 5 \}, \{ 1, 6, 8 \}, \{ 1, 7, 8 \}, \{ 2, 3, 8 \}, \{ 2, 4, 6 \}, \{ 2, 5, 7 \}, \{ 3, 4, 7 \}, \{ 3, 5, 6 \}  &  D_{4}\\\hline
      
      8.26 & 8   &      \{ 1, 4, 5 \}, \{ 1, 5, 7 \}, \{ 1, 6, 8 \}, \{ 2, 3, 8 \}, \{ 2, 4, 8 \}, \{ 2, 6, 7 \}, \{ 3, 4, 7 \}, \{ 3, 5, 6 \}  &  \mathbb{Z}_{2}\\\hline
      
      8.27 & 8   &    \{ 1, 3, 5 \}, \{ 1, 5, 7 \}, \{ 1, 6, 8 \}, \{ 2, 4, 6 \}, \{ 2, 4, 8 \}, \{ 2, 7, 8 \}, \{ 3, 4, 7 \}, \{ 3, 5, 6 \}   &  \mathbb{Z}_{2} \times \mathbb{Z}_{2}\\\hline
      
      8.28 & 8   &       \{ 1, 2, 6 \}, \{ 1, 4, 8 \}, \{ 1, 6, 8 \}, \{ 2, 5, 6 \}, \{ 2, 5, 7 \}, \{ 3, 4, 7 \}, \{ 3, 4, 8 \}, \{ 3, 5, 7 \}    & D_{8} \\\hline
      
      8.29 & 8   &      \{ 1, 3, 8 \}, \{ 1, 4, 8 \}, \{ 1, 7, 8 \}, \{ 2, 3, 7 \}, \{ 2, 4, 7 \}, \{ 2, 5, 6 \}, \{ 3, 5, 6 \}, \{ 4, 5, 6 \}    &  \mathbb{Z}_{2}\times D_{4}\\\hline
      
      8.30 & 8   &       \{ 1, 3, 7 \}, \{ 1, 4, 7 \}, \{ 1, 6, 8 \}, \{ 2, 3, 8 \}, \{ 2, 4, 8 \}, \{ 2, 5, 7 \}, \{ 3, 5, 6 \}, \{ 4, 5, 6 \}   &  \mathbb{Z}_{2} \times \mathbb{Z}_{2} \times S_{3}\\\hline
      
      8.31 & 8   &       \{ 1, 2, 5 \}, \{ 1, 2, 6 \}, \{ 1, 7, 8 \}, \{ 2, 7, 8 \}, \{ 3, 4, 7 \}, \{ 3, 4, 8 \}, \{ 3, 5, 6 \}, \{ 4, 5, 6 \}  & ((\mathbb{Z}_{2}\times \mathbb{Z}_{2} \times \mathbb{Z}_{2}):\mathbb{Z}_{4}):\mathbb{Z}_{2} \\\hline
  
\end{array}
\end{math}
\vspace{5mm}
\caption{All possible symmetric configurations with $s=3$ and $4\leq n \leq 8$}
\end{center}

\end{table}

\section{Future work}

We finish by giving some insight into future avenues of exploration. 

Firstly, in \cite{peet2020coverings}, the concept of an orbiconfiguration was given. This was a generalization of a quotient space of a configuration of order $1$ under a group action. 

It is clear that once again this can be generalaized to higher orders by simply changing the definition as regards to only one line per pair of points to only one plane per triple of points.

Secondly, the work of this paper can be directly continued to investigate the following:

\begin{enumerate}
    \item Count the number of configurations of various orders
    \item Extend notions such as point-transitive, plane-transitive, and flag-transitive configurations
    \item Consider the equivalent of triangle-free configurations - in the case of order $2$, tetrahedron-free configurations of order $2$
\end{enumerate}

Simply put, all the avenues of research interest in configurations can be very simply extended to configurations of higher orders.

Thirdly, it can be investigated what configurations of order $2$ are such that the points and lines (as given by the faces of index $1$) form a configuration of order $1$, and the lines and planes form a configuration of order $1$. Preliminarily, we define such a configuration of order $2$ as a superconfiguration of order $2$. Note that the tetrahedron example (and indeed and platonic solid) form a superconfiguration, but also Example 5.1.

\bibliographystyle{unsrt}  
\bibliography{references} 

\begin{thebibliography}{1}

\bibitem{grunbaum2009configurations}
Branko Gr{\"u}nbaum.
\newblock {\em Configurations of points and lines}, volume 103.
\newblock American Mathematical Soc., 2009.

\bibitem{glynn2007note}
David~G Glynn.
\newblock A note on nk configurations and theorems in projective space.
\newblock {\em Bulletin of the Australian Mathematical Society}, 76(1):15--31,
  2007.

\bibitem{hatcher2005algebraic}
Allen Hatcher.
\newblock {\em Algebraic topology}.
\newblock 2005.

\bibitem{lee2010introduction}
John Lee.
\newblock {\em Introduction to topological manifolds}, volume 202.
\newblock Springer Science \& Business Media, 2010.

\bibitem{mcmullen_schulte_2002}
Peter McMullen and Egon Schulte.
\newblock {\em Abstract Regular Polytopes}.
\newblock Encyclopedia of Mathematics and its Applications. Cambridge
  University Press, 2002.

\bibitem{harary1969graph}
Frank Harary.
\newblock {\em Graph theory}.
\newblock Narosa Publishing House, 1969.

\bibitem{imrich2008topics}
Wilfried Imrich, Sandi Klavzar, and Douglas~F Rall.
\newblock {\em Topics in graph theory: Graphs and their Cartesian product}.
\newblock CRC Press, 2008.

\bibitem{Peet_Configurations_of_higher_2022}
Benjamin Peet.
\newblock {Configurations of higher orders GAP code}.
\newblock {\em https://github.com/benjaminpeet/configurationsofhigherorders},
  2022.

\bibitem{peet2020coverings}
Benjamin Peet.
\newblock Coverings of configurations, prime configurations, and
  orbiconfigurations.
\newblock {\em Revista Colombiana de Matem{\'a}ticas}, 54(2):141--160, 2020.

\end{thebibliography}

\end{document}